\documentclass{amsart}

\usepackage{amssymb,enumitem,mathtools,mleftright,tikz}
\usepackage[initials]{amsrefs}
\usepackage{hyperref}
\usetikzlibrary{cd,decorations.pathmorphing}

\setlist[itemize]{leftmargin = *}
\setlist[enumerate]{leftmargin = *}

\allowdisplaybreaks

\theoremstyle{proclaim}
\newtheorem{thm}{Theorem}[section]

\newtheorem{lem}[thm]{Lemma}
\newtheorem{prop}[thm]{Proposition}

\newtheorem*{lem*}{Lemma}

\theoremstyle{statement}

\newtheorem{rem}[thm]{Remark}
\newtheorem{defn}[thm]{Definition}

\newcommand{\C}{\mathbb C}
\newcommand{\R}{\mathbb{R}}
\newcommand{\ad}{\operatorname{Ad}}
\newcommand{\BB}{\mathcal{B}}
\newcommand{\Br}[1]{\mleft( #1 \mright)}

\newcommand{\Co}[1]{\func{C_{0}}{#1}}
\newcommand{\df}{\coloneqq}

\newcommand{\Id}{\operatorname{id}}
\newcommand{\KK}{\mathcal{K}}
\newcommand{\LL}{\mathcal{L}}
\newcommand{\lt}{\textnormal{lt}}
\newcommand{\rt}{\textnormal{rt}}
\newcommand{\xt}{\otimes}
\newcommand{\aut}{\operatorname{Aut}}

\newcommand{\Mod}{\operatorname{Mod}}
\newcommand{\Seq}[2]{\Br{#1}_{#2}}
\newcommand{\Comm}[2]{\SqBr{#1,#2}}
\newcommand{\func}[2]{#1 \Br{#2}}
\newcommand{\Func}[2]{\func{\Br{#1}}{#2}}
\newcommand{\FUNC}[2]{\func{\SqBr{#1}}{#2}}
\newcommand{\Pair}[2]{\Br{#1,#2}}
\newcommand{\Quad}[4]{\Br{#1,#2,#3,#4}}
\newcommand{\SqBr}[1]{\mleft[ #1 \mright]}
\newcommand{\Trip}[3]{\Br{#1,#2,#3}}
\newcommand{\clspn}{\overline{\operatorname{span}}}
\newcommand{\Cstar}[1]{C^{\ast} (#1)}
\newcommand{\Inner}[2]{\mleft\langle #1,#2 \mright\rangle}
\newcommand{\AltSeq}[2]{(#1)_{#2}}
\newcommand{\Altfunc}[2]{#1(#2)}
\newcommand{\AltQuad}[4]{(#1,#2,#3,#4)}
\newcommand{\AltTrip}[3]{(#1,#2,#3)}
\newcommand{\Schwartz}[1]{\func{\mathcal{S}}{#1}}

\newcommand{\lemref}[1]{Lemma \textup{\ref{#1}}}

\newcommand{\secref}[1]{Section \textup{\ref{#1}}}
\newcommand{\propref}[1]{Proposition \textup{\ref{#1}}}

\renewcommand{\H}{\mathcal{H}}
\renewcommand{\L}[1]{\func{L^{2}}{#1}}

\title{The Modular Stone-von Neumann Theorem}

\author[Hall]{Lucas Hall}
\address{School of Mathematical and Statistical Sciences \\ Arizona State University \\ Tempe, Arizona 85287}
\email{lhall10@asu.edu}

\author[Huang]{Leonard Huang}
\address{Department of Mathematics \& Statistics \\ University of Nevada, Reno, Reno, NV 89557}
\email{LeonardHuang@unr.edu}

\author[Quigg]{John Quigg}
\address{School of Mathematical and Statistical Sciences \\ Arizona State University \\ Tempe, Arizona 85287}
\email{quigg@asu.edu}

\subjclass[2000]{Primary 46L55; Secondary 22D25, 22D35, 43A65, 46L06, 81R15, 81S05}
\keywords{Crossed product, action, coaction, $ C^{\ast} $-correspondence, Morita equivalence, nonabelian duality, Stone-von Neumann Theorem}

\begin{document}

\maketitle

\begin{abstract}
In this paper, we use the tools of nonabelian duality to formulate and prove a far-reaching generalization of the Stone-von Neumann Theorem to modular representations of actions and coactions of locally compact groups on elementary $ C^{\ast} $-algebras. This greatly extends the Covariant Stone-von Neumann Theorem for Actions of Abelian Groups recently proven by L. Ismert and the second author. Our approach is based on a new result about Hilbert $ C^{\ast} $-modules that is simple to state yet is widely applicable and can be used to streamline many previous arguments, so it represents an improvement --- in terms of both efficiency and generality --- in a long line of results in this area of mathematical physics that goes back to J. von Neumann's proof of the classical Stone-von Neumann Theorem.
\end{abstract}

\section*{Introduction} 

In a groundbreaking 1925 paper (\cite{Heisenberg}), W. Heisenberg introduced the very first version of quantum mechanics, known as ``matrix mechanics''. According to matrix mechanics, classical observables (continuous real-valued functions on a phase space describing observable properties of classical particles) were to be replaced by quantum observables (self-adjoint unbounded operators on a Hilbert space describing observable properties of quantum particles), and all quantum phenomena could be mathematically explained by the noncommutativity of certain pairs of quantum observables, as expressed by the Heisenberg Commutation Relation (HCR):
$$
    \Comm{A}{B}
\df A B - B A
=   i \hbar \cdot \mathbf{1}.
$$
A well-known example of a pair of quantum observables that satisfies the HCR is $ \Pair{X}{P} $, where $ X $ and $ P $ are the position and momentum operators on the Hilbert space $ \L{\R} $, defined by
\begin{align*}
\Func{X \psi}{x} & \df x \psi(x), \\
\Func{P \psi}{x} & \df - i \hbar \func{\psi'}{x}.
\end{align*}
Note that $ X $ and $ P $ are not bounded operators on $ \L{\R} $. Rather, they are unbounded operators on $ \L{\R} $, where we may take the Schwartz space $ \Schwartz{\R} $ 
(for example)
to be a suitable dense common domain.

In 1926, E. Schr\"{o}dinger published his famous equation that formed the cornerstone of a second version of quantum mechanics, known as ``wave mechanics'' (\cite{Schrodinger1}). As matrix mechanics and wave mechanics made the same experimental predictions, a mathematical equivalence between them was immediately suspected. Schr\"{o}dinger attempted to prove this equivalence in \cite{Schrodinger2}, but his proof failed to meet the standards of mathematical rigor due to its inability to address technical domain issues surrounding unbounded operators.

To overcome the challenges faced by Schr\"{o}dinger, one may take any pair $ \Pair{A}{B} $ of self-adjoint unbounded operators on a Hilbert space $ \H $ that satisfies the HCR and use Stone's Theorem to exponentiate them, thus obtaining (strongly continuous) one-parameter unitary groups $ U $ and $ V $ on $ \H $:
$$
U = \AltSeq{e^{i x A}}{x \in \R}, \qquad
V = \AltSeq{e^{i y B}}{y \in \R}.
$$
It can then be shown that the pair $ \Pair{U}{V} $ satisfies the so-called ``Weyl Commutation Relation (WCR)'' on $ \H $:
$$
\forall x,y \in \R: \quad
V_{y} U_{x} = e^{i x y} U_{x} V_{y}.
$$
The Stone-von Neumann Theorem --- formulated by M. Stone in \cite{Stone} and rigorously demonstrated by J. von Neumann in \cite{vonNeumann} --- states that $ \Pair{U}{V} $ can be completely described in terms of the pair $ \Pair{\lambda}{\Mod} $, where $ \lambda $ and $ \Mod $ are unitary groups on $ \L{\R} $ obtained by exponentiating (via Stone's Theorem) the pair $ \Pair{X}{P} $. These are called, respectively, the ``left regular representation'' and the ``phase modulation'' of $ \R $. The theorem rigorously proves the equivalence between matrix mechanics and wave mechanics, and is a fundamental result in the mathematical theory of nonrelativistic quantum mechanics.

The following is a rather general version of the Stone-von Neumann Theorem.



\begin{thm}[Stone-von Neumann Theorem \cite{vonNeumann}]
Let $ \Pair{U}{V} $ be a pair of unitary representations of $ \R^{n} $ on a Hilbert space $ \mathcal{H} $ that satisfies the WCR, i.e,
$$
\forall x,y \in \R^{n}: \quad
V_{y} U_{x} = e^{i \Inner{x}{y}} U_{x} V_{y}.
$$
Let $ \lambda $ and $ \Mod $ be the unitary representations of $ \R^{n} $ on the Hilbert space $ \L{\R^{n}} $ defined by
$$
\func{\lambda_{x}}{f} \df \func{f}{\cdot - x} \quad \text{and} \quad
\func{\Mod_{y}}{f} \df e^{i \Inner{\cdot}{y}} f.
$$
Then $ \Pair{\lambda}{\Mod} $ satisfies the WCR on $ \L{\R^{n}} $.
Moreover,
there exist a set $ I $ and a unitary map $ W: \H \to \bigoplus_{i \in I} \L{\R^{n}} $ such that for all $ x,y \in \R^{n} $,
$$
W U_{x} W^{- 1} = \bigoplus_{i \in I} \lambda_{x} \quad \text{and} \quad
W V_{y} W^{- 1} = \bigoplus_{i \in I} \Mod_{y}.
$$
\end{thm}


The Stone-von Neumann Theorem thus completely classifies --- up to unitary equivalence --- all pairs of unitary Hilbert-space representations of $ \R^{n} $ that satisfy the WCR.

In \cite{Mackey}, G. Mackey noted that as $ \R^{n} $ is its own Pontryagin dual, the WCR could be formulated for locally compact abelian groups in general. He then proceeded to formulate and establish the Stone-von Neumann Theorem for the class of all separable locally compact abelian groups, where the separability hypothesis was required by the measure-theoretic machinery of his proof. Shortly after, in \cite{Loomis}, L.H. Loomis was able to remove this superfluous hypothesis with a new proof that did not rely on measure theory. Finally, in \cite{Rieffel1}, M. Rieffel found a mostly algebraic proof containing rudiments of the concept of Morita equivalence for $ C^{\ast} $-algebras, which he went on to develop more fully in \cite{Rieffel2}. It is believed by most operator algebraists that the proper way to view the Mackey-Stone-von Neumann Theorem is as a Morita-equivalence result.

To give an efficient formulation of the Mackey-Stone-von Neumann Theorem, let us define Heisenberg and Schr\"{o}dinger representations.



\begin{defn}[\cites{Mackey,Rieffel1}]
Let $ G $ be a locally compact abelian group. Then a \emph{Heisenberg $ G $-rep\-re\-sen\-ta\-tion} is defined as a triple $ \Trip{\H}{U}{V} $ with the following properties:
\begin{enumerate}
\item
$ \H $ is a Hilbert space.

\item
$ U $ is a unitary representation of $ G $ on $ \H $.

\item
$ V $ is a unitary representation of $ \widehat{G} $ on $ \H $.

\item
$ \Pair{U}{V} $ satisfies the WCR on $ \H $: For all $ x \in G $ and $ \varphi \in \widehat{G} $,
$$
V_{\varphi} U_{x} = \func{\varphi}{x} U_{x} V_{\varphi}.
$$
\end{enumerate}
A Heisenberg $ G $-representation $ \Trip{\H}{U}{V} $ is said to be \emph{equivalent} to another one $ \Trip{\H'}{U'}{V'} $ if and only if there is a unitary map $ W: \H \to \H' $ such that for all $ x \in G $ and $ \varphi \in \widehat{G} $,
$$
W U_{x} W^{- 1}       = U'_{x} \quad \text{and} \quad
W V_{\varphi} W^{- 1} = V'_{\varphi},
$$
in which case we write $ \Trip{\H}{U}{V} \simeq \Trip{\H'}{U'}{V'} $. Note that $ \simeq $ is an equivalence relation on the class of all Heisenberg $ G $-representations.
\end{defn}



\begin{defn}[\cites{Mackey,Rieffel1}] \label{Schroedinger representation}
Let $ G $ be a locally compact abelian group. Then the \emph{Schr\"{o}dinger $ G $-representation} is the triple
$$
\AltTrip{\L{G}}{\lambda}{\Mod},
$$
where $ \lambda $ and $ \Mod $ are, respectively, the unitary representations of $ G $ and $ \widehat{G} $ on the Hilbert space $ \L{G} $ defined by:
$$
\func{\lambda_{x}}{f} \df \Altfunc{f}{x^{- 1} \cdot} \quad \text{and} \quad
\func{\Mod_{\varphi}}{f} \df \varphi f.
$$
\end{defn}


\begin{thm}[\cites{Mackey,Rieffel1}] \label{Schroedinger representations are Heisenberg representations}
Let $ G $ be a locally compact abelian group. Then the following statements are true:
\begin{enumerate}
\item
The Schr\"{o}dinger $ G $-representation is a Heisenberg $ G $-representation.

\item
Let $ I $ be a set and $ \Trip{\H_{i}}{U_{i}}{V_{i}}_{i \in I} $ an $ I $-indexed family of Heisenberg $ G $-rep\-re\-sen\-ta\-tions. Then the direct sum
$$
\bigoplus_{i \in I} \Trip{\H_{i}}{U_{i}}{V_{i}} \df \Trip{\bigoplus_{i \in I} \H_{i}}{\bigoplus_{i \in I} U_{i}}{\bigoplus_{i \in I} V_{i}}
$$
is also a Heisenberg $ G $-representation.
\end{enumerate}
\end{thm}



\begin{thm}[Mackey-Stone-von Neumann Theorem \cites{Mackey,Rieffel1}]
Let $ G $ be a locally compact abelian group. Then every Heisenberg $ G $-representation is a multiple of the Schr\"{o}dinger $ G $-representation.
\end{thm}


Several generalizations of the Mackey-Stone-von Neumann Theorem have appeared in the literature over time. They deal with nonregular group representations (\cite{CMS}) or nonabelian duality (\cite{HKR,Palma}). As the theory of Hilbert $ C^{\ast} $-modules was already highly developed, it seemed natural to expect that the theorem could also be generalized to the setting of Hilbert $ C^{\ast} $-modules, but heuristic arguments against the feasibility of such a generalization were made in \cite{BahtSkeide}, owing to the lack of a spectral decomposition theorem for Hilbert $ C^{\ast} $-modules.

In \cite{huangismert}, L. Ismert and the second author proposed an alternative way of viewing the Mackey-Stone-von Neumann Theorem. Their main idea was that the theorem should be viewed as a special case of a far more general result about the uniqueness of representations of abelian $ C^{\ast} $-dynamical systems on Hilbert $ C^{\ast} $-modules. For a $ C^{\ast} $-dynamical system $ \Trip{A}{G}{\alpha} $ with $ A $ an elementary $ C^{\ast} $-algebra (i.e., isomorphic to the compacts on some Hilbert space) and $ G $ abelian, they were able to completely classify, up to unitary equivalence, all quadruples $ \Quad{X}{\pi}{U}{V} $ with the following properties:
\begin{enumerate}
\item
$ X $ is a Hilbert $ A $-module.

\item
$ \Trip{X}{\pi}{U} $ is a covariant representation of $ \Trip{A}{G}{\alpha} $.

\item
$ \Trip{X}{\pi}{V} $ is a covariant representation of $ \AltTrip{A}{\widehat{G}}{\iota} $, where $ \iota $ denotes the trivial action of $ \widehat{G} $ on $ A $.

\item
$ \Pair{U}{V} $ satisfies the WCR on $ X $.
\end{enumerate}

In the present paper, we further promote the paradigm shift in \cite{huangismert}, streamlining many of their arguments to present what may reasonably be thought of as an abstract Stone-von Neumann Theorem, which we state in terms of Hilbert modules over elementary $ C^{\ast} $-algebras. Notably, this abstract characterization makes no mention of groups but largely takes on the flavor of the von Neumann Uniqueness Theorem (cf. \cite{danacrossed}*{Theorem 4.29}). This abstract characterization offers more flexibility, allowing us to consider dynamical systems involving arbitrary locally compact groups, successfully removing the abelian hypothesis through the tools of nonabelian duality.

We provide a brief overview of nonabelian duality in \S \ref{prelim} as preparation for the parallel sections \S \ref{actions} and \S \ref{coactions} involving the Covariant Stone von-Neumann Theorem for Actions and Coactions respectively.

We thank the referee for comments that improved our paper.

\section{Preliminaries} \label{prelim} 

We refer to \cite{lance} and \cite{enchilada} for Hilbert $ C^{\ast} $-modules and $ C^{\ast} $-correspondences. Here we merely record our conventions. Throughout, $ A $ and $ B $ denote $ C^{\ast} $-algebras, and $ X $ and $ Y $ denote Hilbert $ C^{\ast} $-modules. A \emph{$ B $-$ A $ correspondence} is a Hilbert $ A $-module $ X $ equipped with a homomorphism $ \phi_{X}: B \to \func{\LL}{X} $, which we will always assume to be nondegenerate in the sense that
$$
B X = \{ \FUNC{\func{\phi_{X}}{b}}{x} : b \in B, x \in X \} = X.
$$
We also call $ \phi_{X} $ a \emph{representation} of $ B $ on $ X $. Given a $ C $-$ B $ correspondence $ Y $ and a $ B $-$ A $ correspondence $ X $, we write $ Y \xt_{B} X $ for the $ B $-balanced tensor product, which is a $ C $-$ A $ correspondence. Caution: Lance \cite{lance}*{pages 39--41} would write $ Y \xt_{B} X $ for the \emph{algebraic} $ B $-balanced tensor product, and $ Y \xt_{\phi_{X}} X $ for the completion; nowadays, it seems more customary to use both notations for the completion.

A \emph{$ B $-$ A $ imprimitivity bimodule} is a $ B $-$ A $ correspondence $ X $ such that $ \phi_{X} $ is an isomorphism of $ B $ onto $ \func{\KK}{X} $, in which case we write $ \widetilde{X} $ for the \emph{conjugate} $ A $-$ B $ imprimitivity bimodule. We regard a Hilbert space $ H $ as a $ \func{\KK}{H} $-$ \C $ imprimitivity bimodule (caution: for this purpose, the inner product on $ H $ is taken to be linear in the second variable).

An \emph{isomorphism} of $ B $-$ A $ correspondences $ X $ and $ Y $ is a linear bijection $ \Phi: X \to Y $ such that
\begin{enumerate}
\item
$ \Inner{\func{\Phi}{x}}{\func{\Phi}{y}}_{A} = \Inner{x}{y}_{A} $ and

\item
$ \func{\Phi}{b x} = b \func{\Phi}{x} $
\end{enumerate}
for all $ x,y \in X $ and $ b \in B $. Note that it follows from (1) that $ \func{\Phi}{x a} = \func{\Phi}{x} a $ for all $ x \in X $ and $ a \in A $.

Given a Hilbert $ A $-module $ X $ and a Hilbert $ B $-module $ Y $, the \emph{external tensor product} $ X \xt Y $ is a Hilbert $ \Br{A \xt B} $-module. In particular, if $ B = \C $, then we have a Hilbert space $ Y $, so we can regard the Hilbert $ \Br{A \xt \C} $-module $ X \xt Y $ as a Hilbert $ A $-module.

We record the following standard facts for convenient reference:
\begin{lem*}[Calculus of $ C^{\ast} $-correspondences]
Let $ X $, $ \Seq{X_{i}}{i \in I} $ be Hilbert $ A $-modules and $ Y $ an $ A $-$ C $ correspondence. Then
\begin{enumerate}[label = \textup{(\arabic*)}]
\item
$ X \otimes_{A} A \simeq X $ as Hilbert $ A $-modules;

\item
$ \Br{\bigoplus_{i \in I} X_{i}} \otimes_{A} Y \simeq \bigoplus_{i \in I} \Br{X_{i} \otimes_{A} Y} $ as Hilbert $ C $-modules.
\end{enumerate}
Similarly, if the family $ X $, $ \Seq{X_{i}}{i \in I} $ consists of $ B $-$ A $ correspondences, and $ Z $ is a $ D $-$ B $ correspondence, then also
\begin{enumerate}[resume*]
\item
$ B \otimes_{B} X \simeq X $ as $ B $-$ A $ correspondences;

\item
$ Z \otimes_{B} \Br{\bigoplus_{i \in I} X_{i}} \simeq \bigoplus_{i \in I} \Br{Z \otimes_{B} X_{i}} $ as $ D $-$ A $ correspondences.
\end{enumerate}
If $ X $ is a $ B $-$ A $ imprimitivity bimodule (so that the conjugate $ \widetilde{X} $ exists), then also
\begin{enumerate}[resume*]
\item
$ \widetilde{X} \otimes_{B} X \simeq A $ as $ A $-$ A $ imprimitivity bimodules;

\item
$ X \otimes_{A} \widetilde{X} \simeq B $ as $ B $-$ B $ imprimitivity bimodules.
\end{enumerate}
\end{lem*}

\begin{proof}
The proof is routine. Standard references include \cite{lance,tfb,enchilada}.
\end{proof}

In Sections \ref{actions} and \ref{coactions}, we will apply the Abstract Modular Stone-von Neumann Theorem to a couple of situations where the imprimitivity bimodule comes from crossed-product duality. For completeness, we review the necessary background. For the remainder of this section, $ G $ is a locally compact group and $ A,B $ are $ C^{\ast} $-algebras. An \emph{action} $ \Trip{A}{G}{\alpha} $ of $ G $ is a strongly continuous homomorphism $ \alpha: G \to \func{\aut}{A} $. A \emph{covariant representation} of $ \Trip{A}{G}{\alpha} $ is a triple $ \Trip{X}{\pi}{U} $, where $ X $ is a Hilbert $ C^{\ast} $-module, $ \pi: A \to \func{\LL}{X} $ is a nondegenerate representation, and $ U: G \to \func{\LL}{X} $ is a strongly continuous unitary representation satisfying
$$
\pi \circ \alpha_{x} = \ad U_{x} \circ \pi \qquad \text{for} ~ x \in G.
$$
If $ X $ is a $ C^{\ast} $-algebra $ B $, regarded as a Hilbert $ B $-module in the canonical way, then $ \func{\LL}{X} = \func{M}{B} $ and the continuity assumption on $ U $ is with respect to the strict topology.

A \emph{crossed product} of $ \Trip{A}{G}{\alpha} $ is a covariant representation $ \Trip{A \rtimes_{\alpha} G}{i_{A}}{i_{G}} $, where $ A \rtimes_{\alpha} G $ is a $ C^{\ast} $-algebra and the pair $ \Pair{i_{A}}{i_{G}} $ is universal in the sense that for every covariant representation $ \Trip{X}{\pi}{U} $, there is a unique nondegenerate representation $ \pi \times U: A \rtimes_{\alpha} G \to \func{\LL}{X} $, called the \emph{integrated form} of $ \Pair{\pi}{U} $, making the diagram
$$
\begin{tikzcd}
  A \arrow[r,"i_{A}"] \arrow[dr,"\pi"']
& \func{M}{A \rtimes_{\alpha} G} \arrow[d,"\pi \times U","!"',dashed]
& G \arrow[l,"i_{G}"'] \arrow[dl,"U"] \\
& \func{\LL}{X}
\end{tikzcd}
$$
commute.

We identify $ A \xt \Co{G} $ with $ \Co{G,A} $ in the usual way. The regular covariant representation of $ \Trip{A}{G}{\alpha} $ is $ \Trip{A \xt \L{G}}{\Br{\Id \xt M} \circ \widetilde{\alpha}}{1 \xt \lambda} $, where $ M $ is the representation of $ \Co{G} $ on $ \L{G} $ by multiplication operators, $ \lambda $ is the left regular representation of $ G $, and $ \widetilde{\alpha}: A \to \func{M}{A \xt \Co{G}} $ is given by
$$
\FUNC{\func{\func{\widetilde{\alpha}}{a}}{b \xt f}}{x} = \func{f}{x} \func{\alpha_{x^{- 1}}}{a} b \qquad
\text{for} ~ a,b \in A, ~ f \in \Co{G}, ~ x \in G.
$$

Dually, a \emph{coaction} $ \Trip{A}{G}{\delta} $ of $ G $ is a homomorphism $ \delta: A \to \func{M}{A \xt \Cstar{G}} $ satisfying both of the following:
\begin{align*}
\clspn \{ \func{\delta}{A} \Br{1 \xt \Cstar{G}} \} & = A \xt \Cstar{G}, \\
\Br{\delta \xt \Id} \circ \delta                   & = \Br{\Id \xt \delta_{G}} \circ \delta,
\end{align*}
where $ \delta_{G}: \Cstar{G} \to \func{M}{\Cstar{G} \xt \Cstar{G}} $ is the integrated form of the unitary homomorphism $ s \mapsto s \xt s $. Note that $ \delta_{G} $ is a coaction of $ G $ on $ \Cstar{G} $. A \emph{covariant representation} of $ \Trip{A}{G}{\delta} $ is a triple $ \Trip{X}{\pi}{\mu} $, where $ X $ is a Hilbert $ C^{\ast} $-module, $ \pi $ and $ \mu $ are nondegenerate representations of $ A $ and $ \Co{G}$, respectively, on $ X $ satisfying
$$
\Br{\delta \xt \Id} \circ \delta = \ad \Func{\mu \xt \Id}{w_{G}} \circ \delta,
$$
and where $ w_{G} \in \func{M}{\Co{G} \xt \Cstar{G}} $ corresponds to the canonical embedding of $ G $ in $ \func{M}{\Cstar{G}} $. A \emph{crossed product} of $ \Trip{A}{G}{\delta} $ is a covariant representation $(A\rtimes_\delta G,j_A,j_G)$, where $ A \rtimes_{\delta} G $ is a $ C^{\ast} $-algebra and the pair $ \Pair{j_{A}}{j_{G}} $ is universal in the sense that for every covariant representation $ \Trip{X}{\pi}{\mu} $, there is a unique nondegenerate representation $ \pi \times \mu: A \rtimes_{\delta} G \to \func{\LL}{X} $, called the \emph{integrated form} of $ \Pair{\pi}{\mu} $, making the diagram
$$
\begin{tikzcd}
  A \arrow[r,"j_{A}"] \arrow[dr,"\pi"']
& \func{M}{A \rtimes_{\delta} G} \arrow[d,"\pi \times \mu","!"',dashed]
& \Co{G} \arrow[l,"j_{G}"'] \arrow[dl,"\mu"] \\
& \func{\LL}{X}
\end{tikzcd}
$$
commute.

The \emph{regular covariant representation} of $ \Trip{A}{G}{\delta} $ is $ \Trip{A \xt \L{G}}{\Br{\Id \xt \lambda} \circ \delta}{1 \xt M} $.

Raeburn shows in \cite{raeburnrep}*{Example 2.9 (1)} that if $ U: G \to \func{\LL}{X} $ is a unitary representation and $ \mu: \Co{G} \to \func{\LL}{X} $ is a nondegenerate representation, then $ \Trip{X}{U}{\mu} $ is a covariant representation of $ \Trip{\Cstar{G}}{G}{\delta_{G}} $ if and only if $ \Trip{X}{\mu}{U} $ is a covariant representation of $ \Trip{\Co{G}}{G}{\lt} $, where $ \lt $ is the action of $ G $ on $ \Co{G} $ given by left translation.

If $ \Trip{A}{G}{\alpha} $ is an action, then the \emph{dual coaction} $ \widehat{\alpha} $ of $ G $ on $ A \rtimes_{\alpha} G $ is determined by
\begin{align*}
\func{\widehat{\alpha}}{\func{i_{A}}{a}} & = \func{i_{A}}{a} \xt 1 \qquad \text{for} ~ a \in A, \\
\func{\widehat{\alpha}}{\func{i_{G}}{x}} & = \func{i_{G}}{x} \xt x \qquad \text{for} ~ x \in G.
\end{align*}
If $ \Trip{A}{G}{\delta} $ is a coaction, then the \emph{dual action} $ \widehat{\delta} $ of $ G $ on $ A \rtimes_{\delta} G $ is determined by
\begin{align*}
\widehat{\delta}_{x} \circ j_{A} & = j_{A}, \\
\widehat{\delta}_{x} \circ j_{G} & = j_{G} \circ \rt_{x} \qquad \text{for} ~ x \in G,
\end{align*}
where $ \rt $ is the action of $ G $ on $ \Co{G} $ given by right translation.

\section{The Abstract Stone-von Neumann Theorem} 

We now specialize the methods of the preliminaries to the case of $ X $ a nonzero Hilbert $ A $-module, where $ A $ is an elementary $ C^{\ast} $-algebra. Thus $ A \cong \func{\KK}{H} $ for some nonzero Hilbert space $ H $, which is necessarily unique up to isomorphism (see \cite{dixmier}*{Corollary 4.1.6}). Since $ A $ is simple, the Hilbert module $ X $ is full.

A Hilbert space $ H $ is already an $ A $-$ \C $ imprimitivity bimodule, so its conjugate $ \widetilde{H} $ is a $ \C $-$ A $ imprimitivity bimodule. For any Hilbert space $ L $, $ L \otimes_{\C} \widetilde{H} $ is a Hilbert $ A $-module. In fact, every Hilbert $ A $-module is of this form, as we show in the following lemma.

\begin{lem} \label{hs factor}
If $ X $ is any Hilbert $ A $-module, then there is a Hilbert space $ L $ such that $ X \cong L \otimes_{\C} \widetilde{H} $ as Hilbert $ A $-modules. This Hilbert space is unique up to isomorphism.
\end{lem}
 
\begin{proof}
We take $ L = X \xt_{A} H $, which is a Hilbert space since $ H $ is a Hilbert $ \C $-module. In fact, $ H $ is an $ A $-$ \C $ imprimitivity bimodule, and the calculus of imprimitivity bimodules gives
$$
       L \xt_{\C} \widetilde{H}
\simeq X \xt_{A} H \xt_{\C} \widetilde{H}
\simeq X \xt_{A} A
\simeq X.
$$
For uniqueness up to isomorphism, if $ M $ is another Hilbert space such that $ L \xt_{\C} \widetilde{H} \simeq M \xt_{\C} \widetilde{H} $, then similar computations show that tensoring with $ H $ gives $ L \simeq M $.
\end{proof}

\begin{lem} \label{op restrict} 
In \lemref{hs factor}, identify $ X $ with $ L \xt_{\C} \widetilde{H} $. Then the map
$$
b \mapsto b \xt_{\C} 1
$$
from $ \func{\KK}{L} $ to $ \func{\LL}{X} $ is an isomorphism onto $ \func{\KK}{X} $. In this way, \lemref{hs factor} gives an isomorphism of $ \func{\KK}{X} $-$ A $ correspondences.
\end{lem}

\begin{proof}
Since $ H $ is an $ A $-$ \C $ imprimitivity bimodule, the first assertion follows from \cite{lance}*{Proposition 4.7}, and then the second assertion follows immediately.
\end{proof}

\begin{rem}
With the notation above, Schweizer proves in \cite{schweizer}*{Lemma 3} that $ X \simeq \func{\KK}{H,L} $. We can give an alternative proof of this by combining \lemref{hs factor} with the elementary fact that for any $ C^{\ast} $-algebra $ D $ and any two Hilbert $ D $-modules $ X $ and $ Y $, there is a unique isomorphism
$$
\Upsilon: X \xt_{D} \widetilde{Y} \xrightarrow{\simeq} \func{\KK}{Y,X}
$$
of $ \func{\KK}{X} $-$ \func{\KK}{Y} $ correspondences such that
$$
\func{\Upsilon}{x \xt \widetilde{y}} = \theta_{x,y} \qquad \text{for} ~ x \in X, ~ y \in Y.
$$

Actually, Schweizer gets the Hilbert space $ L $ a different way, namely by recognizing that $ \func{\KK}{X} $ is an elementary $ C^{\ast} $-algebra. It is easy to see that his Hilbert space is isomorphic to our $ L $. Related results on Hilbert modules over elementary $ C^{\ast} $-algebras are contained in, for example, \cite{bakgulcompact,magajna}.
\end{rem}

\begin{defn}
For any two $ B $-$ A $ correspondences $ X $ and $ Y $, we say that $ Y $ is a \emph{multiple} of $ X $ provided $ Y \cong \bigoplus_{I} X $ for some set $ I $.
\end{defn}

\begin{prop}[Abstract Modular Stone-von Neumann Theorem] \label{The Abstract Modular Stone-von Neumann Theorem}
If $ A $ is elementary and $ X $ a $ B $-$ A $ imprimitivity bimodule, then every $ B $-$ A $ correspondence is a multiple of $ X $.
\end{prop}

\begin{proof}
Let $ Y $ be a $ B $-$ A $ correspondence. By Lemmas \ref{hs factor} and \ref{op restrict}, without loss of generality,
$$
X = L \xt_{\C} \widetilde{H} \quad \text{and} \quad
Y = M \xt_{\C} \widetilde{H},
$$
where $ H $, $ M $, and $ L $ are Hilbert spaces, with $ B $ acting irreducibly on $ L $ and nondegenerately on $ M $. Since $ B $ is elementary, its image in $ \func{\BB}{L} $ is the algebra of compact operators $ \func{\KK}{L} $. Thus the nondegenerate representation of $ B $ on $ M $ is a multiple of the irreducible representation on $ L $. In other words, for some set $ I $, we have $ M \simeq \bigoplus_{I} L $ as $ B $-$ \C $ correspondences.

Then by the calculus of $ C^{\ast} $-correspondences,
\begin{align*}
         M \xt_{\C} \widetilde{H}
& \simeq \Br{\bigoplus_{I} L} \xt_{\C} \widetilde{H} \\
& \simeq \bigoplus_{I} (L \xt_{\C} \widetilde{H}) \\
& \simeq \bigoplus_{I} X
\end{align*}
as $ B $-$ A $ correspondences, and we are done.
\end{proof}

\begin{rem}
\propref{The Abstract Modular Stone-von Neumann Theorem} could be used to help classify Hilbert modules over elementary $ C^{\ast} $-algebras. However, in fact \cite{subgroup}*{Remark 6.5} gives a classification where the coefficient algebra can be any direct sum of elementary $ C^{\ast} $-algebras.
\end{rem}

\section{The Covariant Stone-von Neumann Theorem for Actions} \label{actions} 

The Covariant Stone-von Neumann Theorem for Actions of Abelian Groups formulated and proven in \cite{huangismert} generalizes the Mackey-Stone-von Neumann Theorem by replacing Hilbert spaces by Hilbert $ C^{\ast} $-modules as the basis of representations. However, the theorem is more than just a statement about representations of a locally compact abelian group $ G $ and its dual $ \widehat{G} $ on a Hilbert $ C^{\ast} $-module, as the WCR alone suggests. As a Hilbert $ C^{\ast} $-module is a right module $ X $ over a $ C^{\ast} $-algebra $ A $, one can introduce a strongly continuous action $ \alpha $ of $ G $ on $ A $ and attempt to classify ``representations'' of the $ C^{\ast} $-dynamical system $ \Trip{A}{G}{\alpha} $ on $ X $, assuming one knows precisely what these ``representations'' are. The next definition seeks to realize this assumption.



\begin{defn}[\cite{huangismert}]
Let $ \Trip{A}{G}{\alpha} $ be a $ C^{\ast} $-dynamical system with $ G $ abelian. Then an \emph{abelian Heisenberg $ \Trip{A}{G}{\alpha} $-modular representation} is defined as a quadruple $ \Quad{X}{\pi}{U}{V} $ with the following properties:
\begin{enumerate}
\item
$ X $ is a Hilbert $ A $-module.

\item \label{pi U}
$ \Trip{X}{\pi}{U} $ is a covariant representation of $ \Trip{A}{G}{\alpha} $.

\item \label{pi V}
$ \Trip{X}{\pi}{V} $ is a covariant representation of $ \AltTrip{A}{\widehat{G}}{\iota} $, where $ \iota $ denotes the trivial action of $ \widehat{G} $ on $ A $.

\item
$ \Pair{U}{V} $ satisfies the WCR on $ X $: For all $ x \in G $ and $ \varphi \in \widehat{G} $,
$$
V_{\varphi} U_{x} = \func{\varphi}{x} U_{x} V_{\varphi}.
$$
\end{enumerate}
An abelian Heisenberg $ \Trip{A}{G}{\alpha} $-modular representation $ \Quad{X}{\pi}{U}{V} $ is said to be \emph{equivalent} to another one $ \Quad{X'}{\pi'}{U'}{V'} $ if and only if there is a unitary map $ W: X \to X' $ such that for all $ a \in A $, $ x \in G $, and $ \varphi \in \widehat{G} $,
$$
W \func{\pi}{a} W^{- 1} = \func{\pi'}{a}, \quad
W U_{x}         W^{- 1} = U'_{x},         \quad \text{and} \quad
W V_{\varphi}   W^{- 1} = V'_{\varphi},
$$
in which case we write $ \Quad{X}{\pi}{U}{V} \simeq \Quad{X'}{\pi'}{U'}{V'} $. Note that $ \simeq $ is an equivalence relation on the class of all abelian Heisenberg $ \Trip{A}{G}{\alpha} $-modular representations.
\end{defn}


Note that in the definition above, in addition to the WCR, there are two commutation relations coming from the covariant representations in Properties \eqref{pi U} and \eqref{pi V}. These two relations are novel features of the Covariant Stone-von Neumann Theorem for Actions of Abelian Groups that do not arise in the Mackey-Stone-von Neumann Theorem.



\begin{defn}[\cite{huangismert}]
Let $ \Trip{A}{G}{\alpha} $ be a $ C^{\ast} $-dynamical system with $ G $ abelian. Then the \emph{abelian Schr\"{o}dinger $ \Trip{A}{G}{\alpha} $-modular representation} is the quadruple
$$
\AltQuad{A \otimes \L{G}}{\Br{\Id \otimes M} \circ \widetilde{\alpha}}{1 \otimes \lambda}{1 \otimes \Mod},
$$
where $ \AltTrip{A \otimes \L{G}}{\Br{\Id \xt M} \circ \widetilde{\alpha}}{1 \xt \lambda} $ is the regular covariant representation of $ \Trip{A}{G}{\alpha} $ (defined in the introduction) and $ \Mod $ is the phase modulation of $ G $, as in Definition \ref{Schroedinger representation}.
\end{defn}


The following theorem is reminiscent of Theorem \ref{Schroedinger representations are Heisenberg representations}.



\begin{thm}[\cite{huangismert}]
Let $ \Trip{A}{G}{\alpha} $ be a $ C^{\ast} $-dynamical system with $ G $ abelian. Then the following statements are true:
\begin{enumerate}[label = \textup{(\arabic*)}]
\item
The abelian Schr\"{o}dinger $ \Trip{A}{G}{\alpha} $-modular representation is an abelian Heisenberg $ \Trip{A}{G}{\alpha} $-modular representation.

\item
Let $ I $ be a set and $ \Quad{X_{i}}{\pi_{i}}{U_{i}}{V_{i}}_{i \in I} $ an $ I $-indexed family of abelian Heisenberg $ \Trip{A}{G}{\alpha} $-modular representations. Then the direct sum
$$
\bigoplus_{i \in I} \Quad{X_{i}}{\pi_{i}}{U_{i}}{V_{i}}_{i \in I} \df \Quad{\bigoplus_{i \in I} X_{i}}{\bigoplus_{i \in I} \pi_{i}}{\bigoplus_{i \in I} U_{i}}{\bigoplus_{i \in I} V_{i}}
$$
is also an abelian Heisenberg $ \Trip{A}{G}{\alpha} $-modular representation.
\end{enumerate}
\end{thm}


The Covariant Stone-von Neumann Theorem for Actions of Abelian Groups can now be stated concisely as follows.



\begin{thm}[Covariant Stone-von Neumann Theorem for Actions of Abelian Groups \cite{huangismert}]
Let $ \Trip{A}{G}{\alpha} $ be a $ C^{\ast} $-dynamical system with $ A $ an elementary $ C^{\ast} $-algebra and $ G $ abelian. Then every abelian Heisenberg $ \Trip{A}{G}{\alpha} $-modular representation is a multiple of the abelian Schr\"{o}dinger $ \Trip{A}{G}{\alpha} $-modular representation.
\end{thm}


In order to generalize the Covariant Stone-von Neumann Theorem for Actions of Abelian Groups to include actions of nonabelian groups, we must first reformulate the WCR in nonabelian terms.

Let $ \Trip{A}{G}{\alpha} $ be a $ C^{\ast} $-dynamical system with $ G $ abelian, and let the following data be given:
\begin{itemize}
\item
A Hilbert $ A $-module $ X $.

\item
A unitary representation $ U $ of $ G $ on $ X $.

\item
A unitary representation $ V $ of $ \widehat{G} $ on $ X $.

\item
A representation $ \pi $ of $ A $ by adjointable operators on $ X $.
\end{itemize}
Let $ \widetilde{V} $ denote the integrated form of $ V $, which is a representation of $ \Cstar{\widehat{G}} $ on $ X $. Letting $ \mathcal{F}: \Cstar{\widehat{G}} \to \Co{G} $ denote the $ \widehat{G} $-Fourier transform, we get a representation $ \mu \df \widetilde{V} \circ \mathcal{F}^{- 1} $ of $ \Co{G} $ on $ X $. It is a standard consequence of properties of Fourier transforms that $ \Pair{U}{V} $ satisfies the WCR if and only if $ \Trip{X}{\mu}{U} $ is a covariant representation of $ \Trip{\Co{G}}{G}{\lt} $ (see \cite{danacrossed}*{Lemma 4.28}, for example).

This reformulation of the WCR does not appeal to the abelian nature of $ G $ and can thus be applied even to nonabelian groups. Moreover, $ \Trip{X}{\pi}{V} $ is a covariant representation of $ \AltTrip{\widehat{G}}{A}{\iota} $ if and only if $ \pi $ and $ \mu $ commute.

These observations motivate the following alternative definition of a modular representation of a $ C^{\ast} $-dynamical system, which makes sense even if the underlying group is not abelian.



\begin{defn} \label{Heisenberg modular representation}
Let $ \Trip{A}{G}{\alpha} $ be a $ C^{\ast} $-dynamical system. Then a \emph{Heisenberg $ \Trip{A}{G}{\alpha} $-modular representation} is a quadruple $ \Quad{X}{\pi}{U}{\mu} $ with the following properties:
\begin{enumerate}
\item
$ X $ is a Hilbert $ A $-module.

\item
$ \Trip{X}{\pi}{U} $ is a covariant representation of $ \Trip{A}{G}{\alpha} $.

\item
$ \Trip{X}{\mu}{U} $ is a covariant representation of $ \Trip{\Co{G}}{G}{\lt} $.

\item
$ \pi $ and $ \mu $ commute.
\end{enumerate}
\end{defn}



\begin{lem} \label{Bijection between modular and covariant representations}
Let $ \Trip{A}{G}{\alpha} $ be a $ C^{\ast} $-dynamical system, and let $ \widehat{\alpha} $ denote the dual coaction of $ G $ on $ A \rtimes_{\alpha} G $. Then
$$
\Quad{X}{\pi}{U}{\mu} \mapsto \Trip{X}{\pi \times U}{\mu}
$$
gives a bijection from the Heisenberg $ \Trip{A}{G}{\alpha} $-modular representations to the covariant representations of $ \Trip{A \rtimes_{\alpha} G}{G}{\widehat{\alpha}} $. Consequently,
$$
\Quad{X}{\pi}{U}{\mu} \mapsto \Pair{X}{\Br{\pi \times U} \times \mu}
$$
gives a bijection from the Heisenberg $ \Trip{A}{G}{\alpha} $-modular representations to the nondegenerate representations of $ \Br{A \rtimes_{\alpha} G} \rtimes_{\widehat{\alpha}} G $.
\end{lem}

\begin{proof}
By \cite{enchilada}*{Proposition A.63} and the discussion preceding it (see also \cite{raeburnrep}*{Examples 2.9}), $ \Quad{X}{\pi}{U}{\mu} $ is a Heisenberg $ \Trip{A}{G}{\alpha} $-modular representation if and only if $ \Trip{X}{\pi \times U}{\mu} $ is a covariant representation of $ \Trip{A \rtimes_{\alpha} G}{G}{\widehat{\alpha}} $. Since the latter is equivalent to
$ \Pair{X}{\Br{\pi \times U} \times \mu} $ being a nondegenerate representation of $ \Br{A \rtimes_{\alpha} G} \rtimes_{\widehat{\alpha}} G $ by definition, we are done.
\end{proof}



\begin{defn}
A Heisenberg $ \Trip{A}{G}{\alpha} $-modular representation $ \Quad{X}{\pi}{U}{\mu} $ is said to be \emph{equivalent} to another one $ \Quad{X'}{\pi'}{U'}{\mu'} $ if and only if there is a unitary map $ W: X \to X' $ such that for all $ a \in A $, $ x \in G $, and $ f \in \Co{G} $,
$$
W \func{\pi}{a} W^{- 1} = \func{\pi'}{a}, \quad
W U_{x}         W^{- 1} = U'_{x},         \quad \text{and} \quad
W \func{\mu}{f} W^{- 1} = \func{\mu'}{f},
$$
in which case we write $ \Quad{X}{\pi}{U}{\mu} \simeq \Quad{X'}{\pi'}{U'}{\mu'} $.
\end{defn}


Note that $ \simeq $ is an equivalence relation on the class of all Heisenberg $ \Trip{A}{G}{\alpha} $-modular representations. Heisenberg modular representations are equivalent if and only if the associated representations of the double crossed product are equivalent.



\begin{defn}
Let $ \Trip{A}{G}{\alpha} $ be a $ C^{\ast} $-dynamical system. Then the \emph{Schr\"{o}dinger $ \Trip{A}{G}{\alpha} $-modular representation} is the quadruple
$$
\AltQuad{A \otimes \L{G}}{\Br{\Id \otimes M} \circ \widetilde{\alpha}}{1 \otimes \lambda}{1 \otimes M}.
$$
\textbf{Note:} If $ G $ is abelian, then $ M = \mu_{\Mod} $.
\end{defn}



\begin{prop} \label{Schroedinger modular representations are Heisenberg modular representations}
Let $ \Trip{A}{G}{\alpha} $ be a $ C^{\ast} $-dynamical system. Then the following statements are true:
\begin{enumerate}[label=\textup{(\arabic*)}]
\item
The Schr\"{o}dinger $ \Trip{A}{G}{\alpha} $-modular representation is a Heisenberg $ \Trip{A}{G}{\alpha} $-modular representation.

\item
Let $ I $ be a set and $ \Quad{X_{i}}{\pi_{i}}{U_{i}}{\mu_{i}}_{i \in I} $ an $ I $-indexed family of Heisenberg $ \Trip{A}{G}{\alpha} $-modular representations. Then the direct sum
$$
\bigoplus_{i \in I} \Quad{X_{i}}{\pi_{i}}{U_{i}}{\mu_{i}} \df \Quad{\bigoplus_{i \in I} X_{i}}{\bigoplus_{i \in I} \pi_{i}}{\bigoplus_{i \in I} U_{i}}{\bigoplus_{i \in I} \mu_{i}}
$$
is also a Heisenberg $ \Trip{A}{G}{\alpha} $-modular representation.
\end{enumerate}
\end{prop}

\begin{proof}
Statement (1) is true by virtue of the following facts:
\begin{itemize}
\item
The regular covariant representation of $ \Trip{A}{G}{\alpha} $ is precisely
$$
\AltTrip{A \otimes \L{G}}{\Br{\Id \otimes M} \circ \widetilde{\alpha}}{1 \otimes \lambda}.
$$

\item
$ \Trip{\L{G}}{M}{\lambda} $ is a covariant representation of $ \Trip{\Co{G}}{G}{\lt} $.

\item
$ \Br{\Id \otimes M} \circ \widetilde{\alpha} $ commutes with $ 1 \otimes M $ as $ \Co{G} $ is commutative.
\end{itemize}
Statement (2) follows routinely from the definitions.
\end{proof}


By Lemma \ref{Bijection between modular and covariant representations} and Proposition \ref{Schroedinger modular representations are Heisenberg modular representations}, the direct sum of Heisenberg modular representations is associated to the direct sum of their associated representations of the double crossed product. In particular, a Heisenberg modular representation $ \Quad{X}{\pi}{U}{\mu} $ is a multiple of another one $ \Quad{X'}{\pi'}{U'}{\mu'} $ if and only if the associated representation $ \Br{\pi \times U} \times \mu $ of the double crossed product is a multiple of $ \Br{\pi' \times U'} \times \mu' $.



\begin{prop}[Covariant Stone-von Neumann Theorem for Actions]
Let $ \Trip{A}{G}{\alpha} $ be a $ C^{\ast} $-dynamical system with $ A $ an elementary $ C^{\ast} $-algebra. Then every Heisenberg $ \Trip{A}{G}{\alpha} $-modular representation is a multiple of the Schr\"{o}dinger $ \Trip{A}{G}{\alpha} $-modular representation.
\end{prop}

\begin{proof}
Let $ \Quad{X}{\pi}{U}{\mu} $ be a Heisenberg $ \Trip{A}{G}{\alpha} $-modular representation. Imai-Takai Duality says that the integrated form of the Schr\"{o}dinger $ \Trip{A}{G}{\alpha} $-modular representation
$$
\AltQuad{A \otimes \L{G}}{\Br{\Id \otimes M} \circ \widetilde{\alpha}}{1 \otimes \lambda}{1 \otimes M}
$$
is a $ C^{\ast} $-isomorphism from $ \Br{A \rtimes_{\alpha} G} \rtimes_{\widehat{\alpha}} G $ to $ A \otimes \func{\KK}{\L{G}} $. Using this $ C^{\ast} $-isomorphism and the identification
$$
\Altfunc{\KK}{A \otimes \L{G}} \cong A \otimes \Altfunc{\KK}{\L{G}},
$$
the $ \func{\KK}{A \otimes \L{G}} $-$ A $ imprimitivity bimodule $ A \otimes \L{G} $ becomes an $ \Br{\Br{A \rtimes_{\alpha} G} \rtimes_{\widehat{\alpha}} G} $-$ A $ imprimitivity bimodule, and $ \phi_{A \otimes \L{G}} $ becomes the integrated form of the Schr\"{o}dinger $ \Trip{A}{G}{\alpha} $-modular representation. The Covariant Stone-von Neumann Theorem for Actions then follows from the paragraph preceding this proposition.
\end{proof}

\section{The Covariant Stone-von Neumann Theorem for Coactions} \label{coactions} 

In this section, we apply the Abstract Modular Stone-von Neumann Theorem again to crossed-product duality, this time starting with a coaction rather than an action. The development is parallel to that in \secref{actions}, and we significantly streamline the presentation.

\begin{defn}
Given a coaction $ \Trip{A}{G}{\delta} $, a \emph{Heisenberg $ \Trip{A}{G}{\delta} $-modular representation} is a quadruple $ \Quad{X}{\pi}{\mu}{U} $ with the following properties:
\begin{itemize}
\item
$ \Trip{X}{\pi}{\mu} $ is a covariant representation of $ \Trip{A}{G}{\delta} $.

\item
$ \Trip{X}{\mu}{U} $ is a covariant representation of $ \Trip{\Co{G}}{G}{\rt} $.

\item
$ \pi $ and $ U $ commute.
\end{itemize}
\end{defn}

\begin{lem}
With the above notation, the assignments $ \Quad{X}{\pi}{\mu}{U} \mapsto \Trip{X}{\pi \times \mu}{U} $ and $ \Quad{X}{\pi}{\mu}{U} \mapsto \Pair{X}{\Br{\pi \times \mu} \times U} $ give bijections from the Heisenberg $ \Trip{A}{G}{\delta} $-modular representations to the covariant representations of $ \AltTrip{A \rtimes_{\delta} G}{G}{\widehat{\delta}} $ and the nondegenerate representations of $ \Br{A \rtimes_{\delta} G} \rtimes_{\widehat{\delta}} G $, respectively.
\end{lem}

It seems difficult to find the above statement in the literature, but it is an easy application of standard techniques, and for completeness, we include an outline of the proof below.

\begin{proof}
For the first part, it suffices to note that
\begin{itemize}
\item
if $ \Quad{X}{\pi}{\mu}{U} $ is a Heisenberg $ \Trip{A}{G}{\delta} $-module representation, then $ \Trip{X}{\pi \times \mu}{U} $ is a covariant representation of $ \AltTrip{A \rtimes_{\delta} G}{G}{\widehat{\delta}} $, and this can be quickly checked separately on the generators $ \func{\pi}{a} $ and $ \func{\mu}{f} $ for $ a \in A $ and $ f \in \Co{G} $, while, on the other hand,

\item
if $ \Trip{X}{\sigma}{U} $ is a covariant representation of $ \AltTrip{A \rtimes_{\delta} G}{G}{\widehat{\delta}} $, then there is a unique covariant representation $ \Trip{X}{\pi}{\mu} $ of $ \Trip{A}{G}{\delta} $ such that $ \sigma = \pi \times \mu $, and, moreover, covariance of $ \Trip{X}{\sigma}{U} $ applied to the generators $ \func{\pi}{a} $ and $ \func{\mu}{f} $ shows that $ \Trip{X}{\mu}{U} $ is covariant for $ \Trip{\Co{G}}{G}{\rt} $ and that $ \pi $ commutes with $ U $.
\end{itemize}

Then the second part follows immediately since, by definition, the covariant representations of the action $ \widehat{\delta} $ are in bijective correspondence with the nondegenerate representations of the crossed product by $ \widehat{\delta} $.
\end{proof}

\begin{defn}
The \emph{Schr\"{o}dinger $ \Trip{A}{G}{\delta} $-modular representation} is
$$
\AltQuad{A \xt \L{G}}{\Br{\Id \xt \lambda} \circ \delta}{1 \xt M}{1 \xt \rho},
$$
where $ \rho $ is the right regular representation of $ G $.
\end{defn}

\begin{lem}
The Schr\"{o}dinger $ \Trip{A}{G}{\delta} $-modular representation is a Heisenberg modular representation.
\end{lem}

\begin{proof}
This is true because of the following facts:
\begin{itemize}
\item
$ \Trip{A \xt \L{G}}{\Br{\Id \xt \lambda} \circ\delta}{1 \xt M} $ is the regular covariant representation of $ \Trip{A}{G}{\delta} $.

\item
$ \Trip{A \xt \L{G}}{M}{\rho} $ is a covariant representation of $ \Trip{\Co{G}}{G}{\rt} $.

\item
$ \Br{\Id \xt \lambda} \circ \delta $ commutes with $ 1 \xt \rho $ since $ \lambda $ and $ \rho $ commute. \qedhere
\end{itemize}
\end{proof}

Just as we did for actions, we say that a Heisenberg modular representation is a multiple of another one if and only if it is equivalent to a direct sum of copies of it.

\begin{thm}[Covariant Stone-von Neumann Theorem for Coactions]
If $ \delta $ is a maximal coaction of $ G $ on an elementary $ C^{\ast} $-algebra $ A $, then every Heisenberg $ \Trip{A}{G}{\delta} $-modular representation is a multiple of the Schr\"{o}dinger $ \Trip{A}{G}{\delta} $-modular representation.
\end{thm}

\begin{proof}
The Katayama Duality Theorem says that if the coaction $ \delta $ is maximal, then the integrated form of the Schr\"{o}dinger modular representation is a $ C^{\ast} $-isomorphism of $ \Br{A \rtimes_{\delta} G} \rtimes_{\widehat{\delta}} G $ onto $ A \xt \KK $. Using this $ C^{\ast} $-isomorphism, the $ \Br{A \xt \KK} $-$ A $ imprimitivity bimodule $ A \xt \L{G} $ becomes an $ \Br{\Br{A \rtimes_\delta G} \rtimes_{\widehat{\delta}} G} $-$ A $ imprimitivity bimodule, and $ \phi_{A \xt \L{G}} $ is the integrated form of the Schr\"{o}dinger $ \Trip{A}{G}{\delta} $-modular representation. Now the result follows from \propref{The Abstract Modular Stone-von Neumann Theorem}.
\end{proof}

\begin{rem}
In its original formulation, the Katayama Duality Theorem \cite{katayama}*{Theorem 8} (see also \cite{maximal}*{Proposition 2.2}) involved the reduced crossed product by the dual action. The version we quoted is really more of a definition than a theorem (see \cite{maximal}*{Definition 3.1}).
\end{rem}


\end{document}